\documentclass[10pt]{amsart}
\usepackage{amsmath}
\usepackage{amssymb}
\usepackage{amsbsy, amsthm,amstext, amsopn}
\usepackage[all]{xy}
\usepackage{amsfonts}
\usepackage{amscd}
\usepackage{stmaryrd}
\usepackage{color}
\usepackage[numbers, square, sort, comma]{natbib}

\usepackage{graphicx}

\newtheorem{thm}{Theorem} [section]
\newtheorem{lemma}[thm]{Lemma}

\newtheorem{corollary}[thm]{Corollary}
\newtheorem{prop}[thm]{Proposition}

\theoremstyle{definition}

\newtheorem{defn}[thm]{Definition}

\newtheorem{example}[thm]{Example}

\newtheorem{remark}[thm]{Remark}

\begin{document}

\numberwithin{equation}{section}

\newcommand{\hs}{\mbox{\hspace{.4em}}}
\newcommand{\ds}{\displaystyle}
\newcommand{\bd}{\begin{displaymath}}
\newcommand{\ed}{\end{displaymath}}
\newcommand{\bcd}{\begin{CD}}
\newcommand{\ecd}{\end{CD}}

\newcommand{\on}{\operatorname}
\newcommand{\proj}{\operatorname{Proj}}
\newcommand{\bproj}{\underline{\operatorname{Proj}}}

\newcommand{\spec}{\operatorname{Spec}}
\newcommand{\Spec}{\operatorname{Spec}}
\newcommand{\bspec}{\underline{\operatorname{Spec}}}
\newcommand{\pline}{{\mathbf P} ^1}
\newcommand{\aline}{{\mathbf A} ^1}
\newcommand{\pplane}{{\mathbf P}^2}
\newcommand{\aplane}{{\mathbf A}^2}
\newcommand{\coker}{{\operatorname{coker}}}
\newcommand{\ldb}{[[}
\newcommand{\rdb}{]]}

\newcommand{\Sym}{\operatorname{Sym}^{\bullet}}
\newcommand{\Symp}{\operatorname{Sym}}
\newcommand{\Pic}{\bf{Pic}}
\newcommand{\Aut}{\operatorname{Aut}}
\newcommand{\Isom}{\operatorname{Isom}}
\newcommand{\PAut}{\operatorname{PAut}}

\newcommand{\too}{\twoheadrightarrow}
\newcommand{\C}{{\mathbf C}}
\newcommand{\Z}{{\mathbf Z}}
\newcommand{\Q}{{\mathbf Q}}
\newcommand{\Cx}{{\mathbf C}^{\times}}
\newcommand{\Cbar}{\overline{\C}}
\newcommand{\Cxbar}{\overline{\Cx}}
\newcommand{\cA}{{\mathcal A}}
\newcommand{\cS}{{\mathcal S}}
\newcommand{\cV}{{\mathcal V}}
\newcommand{\cM}{{\mathcal M}}
\newcommand{\bA}{{\mathbf A}}
\newcommand{\cB}{{\mathcal B}}
\newcommand{\cC}{{\mathcal C}}
\newcommand{\cD}{{\mathcal D}}
\newcommand{\D}{{\mathcal D}}
\newcommand{\cs}{{\mathbf C} ^*}
\newcommand{\boldc}{{\mathbf C}}
\newcommand{\cE}{{\mathcal E}}
\newcommand{\cF}{{\mathcal F}}
\newcommand{\bF}{{\mathbf F}}
\newcommand{\cG}{{\mathcal G}}
\newcommand{\G}{{\mathbb G}}
\newcommand{\cH}{{\mathcal H}}
\newcommand{\CI}{{\mathcal I}}
\newcommand{\cJ}{{\mathcal J}}
\newcommand{\cK}{{\mathcal K}}
\newcommand{\cL}{{\mathcal L}}
\newcommand{\baL}{{\overline{\mathcal L}}}

\newcommand{\fI}{{\mathfrak I}}
\newcommand{\fJ}{{\mathfrak J}}
\newcommand{\fF}{{\mathfrak F}}
\newcommand{\Mf}{{\mathfrak M}}
\newcommand{\bM}{{\mathbf M}}
\newcommand{\bm}{{\mathbf m}}
\newcommand{\cN}{{\mathcal N}}
\newcommand{\theo}{\mathcal{O}}
\newcommand{\cP}{{\mathcal P}}
\newcommand{\cR}{{\mathcal R}}
\newcommand{\Pp}{{\mathbb P}}
\newcommand{\boldp}{{\mathbf P}}
\newcommand{\boldq}{{\mathbf Q}}
\newcommand{\bbL}{{\mathbf L}}
\newcommand{\cQ}{{\mathcal Q}}
\newcommand{\cO}{{\mathcal O}}
\newcommand{\Oo}{{\mathcal O}}
\newcommand{\cY}{{\mathcal Y}}
\newcommand{\OX}{{\Oo_X}}
\newcommand{\OY}{{\Oo_Y}}
\newcommand{\otY}{{\underset{\OY}{\ot}}}
\newcommand{\otX}{{\underset{\OX}{\ot}}}
\newcommand{\cU}{{\mathcal U}}\newcommand{\cX}{{\mathcal X}}
\newcommand{\cW}{{\mathcal W}}
\newcommand{\boldz}{{\mathbf Z}}
\newcommand{\qgr}{\operatorname{q-gr}}
\newcommand{\gr}{\operatorname{gr}}
\newcommand{\rk}{\operatorname{rk}}
\newcommand{\SH}{{\underline{\operatorname{Sh}}}}
\newcommand{\End}{\operatorname{End}}
\newcommand{\uEnd}{\underline{\operatorname{End}}}
\newcommand{\Hom}{\operatorname{Hom}}
\newcommand{\uHom}{\underline{\operatorname{Hom}}}
\newcommand{\uHomY}{\uHom_{\OY}}
\newcommand{\uHomX}{\uHom_{\OX}}
\newcommand{\Ext}{\operatorname{Ext}}
\newcommand{\bExt}{\operatorname{\bf{Ext}}}
\newcommand{\Tor}{\operatorname{Tor}}

\newcommand{\inv}{^{-1}}
\newcommand{\airtilde}{\widetilde{\hspace{.5em}}}
\newcommand{\airhat}{\widehat{\hspace{.5em}}}
\newcommand{\nt}{^{\circ}}
\newcommand{\del}{\partial}

\newcommand{\supp}{\operatorname{supp}}
\newcommand{\GK}{\operatorname{GK-dim}}
\newcommand{\hd}{\operatorname{hd}}
\newcommand{\id}{\operatorname{id}}
\newcommand{\res}{\operatorname{res}}
\newcommand{\lrar}{\leadsto}
\newcommand{\im}{\operatorname{Im}}
\newcommand{\HH}{\operatorname{H}}
\newcommand{\TF}{\operatorname{TF}}
\newcommand{\Bun}{\operatorname{Bun}}

\newcommand{\F}{\mathcal{F}}
\newcommand{\Ff}{\mathbb{F}}
\newcommand{\nthord}{^{(n)}}
\newcommand{\Gr}{{\mathfrak{Gr}}}

\newcommand{\Fr}{\operatorname{Fr}}
\newcommand{\GL}{\operatorname{GL}}
\newcommand{\gl}{\mathfrak{gl}}
\newcommand{\SL}{\operatorname{SL}}
\newcommand{\ff}{\footnote}
\newcommand{\ot}{\otimes}
\def\Ext{\operatorname {Ext}}
\def\Hom{\operatorname {Hom}}
\def\Ind{\operatorname {Ind}}
\def\bbZ{{\mathbb Z}}

\newcommand{\nc}{\newcommand}
\nc{\ol}{\overline} \nc{\cont}{\on{cont}} \nc{\rmod}{\on{mod}}
\nc{\Mtil}{\widetilde{M}} \nc{\wb}{\overline} 
\nc{\wh}{\widehat}  \nc{\mc}{\mathcal}
\nc{\mbb}{\mathbb}  \nc{\K}{{\mc K}} \nc{\Kx}{{\mc K}^{\times}}
\nc{\Ox}{{\mc O}^{\times}} \nc{\unit}{{\bf \on{unit}}}
\nc{\boxt}{\boxtimes} \nc{\xarr}{\stackrel{\rightarrow}{x}}

\nc{\Ga}{\G_a}
 \nc{\PGL}{{\on{PGL}}}
 \nc{\PU}{{\on{PU}}}

\nc{\h}{{\mathfrak h}} \nc{\kk}{{\mathfrak k}}
 \nc{\Gm}{\G_m}
\nc{\Gabar}{\wb{\G}_a} \nc{\Gmbar}{\wb{\G}_m} \nc{\Gv}{G^\vee}
\nc{\Tv}{T^\vee} \nc{\Bv}{B^\vee} \nc{\g}{{\mathfrak g}}
\nc{\gv}{{\mathfrak g}^\vee} \nc{\BRGv}{\on{Rep}\Gv}
\nc{\BRTv}{\on{Rep}T^\vee}
 \nc{\Flv}{{\mathcal B}^\vee}
 \nc{\TFlv}{T^*\Flv}
 \nc{\Fl}{{\mathfrak Fl}}
\nc{\BRR}{{\mathcal R}} \nc{\Nv}{{\mathcal{N}}^\vee}
\nc{\St}{{\mathcal St}} \nc{\ST}{{\underline{\mathcal St}}}
\nc{\Hec}{{\bf{\mathcal H}}} \nc{\Hecblock}{{\bf{\mathcal
H_{\alpha,\beta}}}} \nc{\dualHec}{{\bf{\mathcal H^\vee}}}
\nc{\dualHecblock}{{\bf{\mathcal H^\vee_{\alpha,\beta}}}}
\newcommand{\ramBun}{{\bf{Bun}}}
\newcommand{\ramBuno}{\ramBun^{\circ}}

\nc{\Buntheta}{{\bf Bun}_{\theta}} \nc{\Bunthetao}{{\bf
Bun}_{\theta}^{\circ}} \nc{\BunGR}{{\bf Bun}_{G_\BR}}
\nc{\BunGRo}{{\bf Bun}_{G_\BR}^{\circ}}
\nc{\HC}{{\mathcal{HC}}}
\nc{\risom}{\stackrel{\sim}{\to}} \nc{\Hv}{{H^\vee}}
\nc{\bS}{{\mathbf S}}
\def\BRep{\operatorname {Rep}}
\def\Conn{\operatorname {Conn}}

\nc{\Vect}{{\operatorname{Vect}}}
\nc{\Hecke}{{\operatorname{Hecke}}}

\newcommand{\ZZ}{{Z_{\bullet}}}
\nc{\HZ}{{\mc H}\ZZ} \nc{\eps}{\epsilon}

\nc{\CN}{\mathcal N} \nc{\BA}{\mathbb A}

 \nc{\BB}{\mathbb B}

\nc{\ul}{\underline}

\nc{\bn}{\mathbf n} \nc{\Sets}{{\on{Sets}}} \nc{\Top}{{\on{Top}}}
\nc{\IntHom}{{\mathcal Hom}}

\nc{\Simp}{{\mathbf \Delta}} \nc{\Simpop}{{\mathbf\Delta^\circ}}

\nc{\Cyc}{{\mathbf \Lambda}} \nc{\Cycop}{{\mathbf\Lambda^\circ}}

\nc{\Mon}{{\mathbf \Lambda^{mon}}}
\nc{\Monop}{{(\mathbf\Lambda^{mon})\circ}}

\nc{\Aff}{{\on{Aff}}} \nc{\Sch}{{\on{Sch}}}

\nc{\bul}{\bullet}
\nc{\module}{{\operatorname{-mod}}}

\nc{\dstack}{{\mathcal D}}

\nc{\BL}{{\mathbb L}}

\nc{\BD}{{\mathbb D}}

\nc{\BR}{{\mathbb R}}

\nc{\BT}{{\mathbb T}}

\nc{\SCA}{{\mc{SCA}}}
\nc{\DGA}{{\mc DGA}}

\nc{\DSt}{{DSt}}

\nc{\lotimes}{{\otimes}^{\mathbf L}}

\nc{\bs}{\backslash}

\nc{\Lhat}{\widehat{\mc L}}

\newcommand{\Coh}{\on{Coh}}

\nc{\QCoh}{QC}
\nc{\QC}{QC}
\nc{\Perf}{\on{Perf}}
\nc{\Cat}{{\on{Cat}}}
\nc{\dgCat}{{\on{dgCat}}}
\nc{\bLa}{{\mathbf \Lambda}}

\nc{\BRHom}{\mathbf{R}\hspace{-0.15em}\on{Hom}}
\nc{\BREnd}{\mathbf{R}\hspace{-0.15em}\on{End}}
\nc{\colim}{\on{colim}}
\nc{\oo}{\infty}
\nc{\Mod}{\on{Mod} }

\nc\fh{\mathfrak h}
\nc\al{\alpha}
\nc\la{\alpha}
\nc\BGB{B\bs G/B}
\nc\QCb{QC^\flat}
\nc\qc{\on{QC}}

\def\w{\wedge}
\nc{\vareps}{\varepsilon}

\nc{\fg}{\mathfrak g}

\nc{\Map}{\on{Map}} \nc{\fX}{\mathfrak X}

\nc{\ch}{\check}
\nc{\fb}{\mathfrak b} \nc{\fu}{\mathfrak u} \nc{\st}{{st}}
\nc{\fU}{\mathfrak U}
\nc{\fZ}{\mathfrak Z}
\nc{\fB}{\mathfrak B}

 \nc\fc{\mathfrak c}
 \nc\fs{\mathfrak s}

\nc\fk{\mathfrak k} \nc\fp{\mathfrak p}
\nc\fq{\mathfrak q}

\nc{\BRP}{\mathbf{RP}} \nc{\rigid}{\text{rigid}}
\nc{\glob}{\text{glob}}

\nc{\cI}{\mathcal I}

\nc{\La}{\mathcal L}

\nc{\quot}{/\hspace{-.25em}/}

\nc\aff{\mathit{aff}}
\nc\BS{\mathbb S}

\nc\Loc{{\mc Loc}}
\nc\Tr{{\on{Tr}}}
\nc\Ch{{\mc Ch}}

\nc\ftr{{\mathfrak {tr}}}
\nc\fM{\mathfrak M}

\nc\Id{\operatorname{Id}}

\nc\bimod{\on{-bimod}}

\nc\ev{\operatorname{ev}}
\nc\coev{\operatorname{coev}}

\nc\pair{\operatorname{pair}}
\nc\kernel{\operatorname{kernel}}

\nc\Alg{\operatorname{Alg}}

\nc\init{\emptyset_{\text{\em init}}}
\nc\term{\emptyset_{\text{\em term}}}

\nc\Ev{\on{Ev}}
\nc\Coev{\on{Coev}}

\nc\es{\emptyset}
\nc\m{\text{\it min}}
\nc\M{\text{\it max}}
\nc\cross{\text{\it cr}}
\nc\tr{\on{tr}}

\nc\perf{\on{-perf}}
\nc\inthom{\mathcal Hom}
\nc\intend{\mathcal End}

\newcommand{\Sh}{\mathit{Sh}}

\nc{\Comod}{\on{Comod}}
\nc{\cZ}{\mathcal Z}

\def\interiorsymbol {\on{int}}

\nc\frakf{\mathfrak f}
\nc\fraki{\mathfrak i}
\nc\frakj{\mathfrak j}
\nc\BP{\mathbb P}
\nc\stab{st}
\nc\Stab{St}

\nc\fN{\mathfrak N}
\nc\fT{\mathfrak T}
\nc\fV{\mathfrak V}

\nc\Ob{\on{Ob}}

\nc\fC{\mathfrak C}
\nc\Fun{\on{Fun}}

\nc\Null{\on{Null}}

\nc\BC{\mathbb C}

\nc\loc{\on{Loc}}

\nc\hra{\hookrightarrow}
\nc\fL{\mathfrak L}
\nc\R{\mathbb R}
\nc\CE{\mathcal E}

\nc\sK{\mathsf K}
\nc\sL{\mathsf L}
\nc\sC{\mathsf C}

\nc\Cone{\text{Cone}}

\nc\fY{\mathfrak Y}
\nc\fe{\mathfrak e}
\nc\ft{\mathfrak t}

\nc\wt{\widetilde}
\nc\inj{\mathit{inj}}
\nc\surj{\mathit{surj}}

\nc\Path{\mathit{Path}}
\nc\Set{\mathit{Set}}
\nc\Fin{\mathit{Fin}}

\nc\cyc{\mathit{cyc}}

\nc\per{\mathit{per}}

\nc\sym{\mathit{symp}}
\nc\con{\mathit{cont}}
\nc\gen{\mathit{gen}}
\nc\str{\mathit{str}}
\nc\rsdl{\mathit{res}}
\nc\impr{\mathit{impr}}
\nc\rel{\mathit{rel}}
\nc\pt{\mathit{pt}}
\nc\naive{\mathit{nv}}
\nc\forget{\mathit{For}}

\nc\sH{\mathsf H}
\nc\sW{\mathsf W}
\nc\sE{\mathsf E}
\nc\sP{\mathsf P}
\nc\sB{\mathsf B}
\nc\sS{\mathsf S}
\nc\fH{\mathfrak H}
\nc\fP{\mathfrak P}
\nc\fW{\mathfrak W}
\nc\fE{\mathfrak E}
\nc\sx{\mathsf x}
\nc\sy{\mathsf y}

\nc\ord{\mathit{ord}}

\nc\sm{\mathit{sm}}

\nc\rhu{\rightharpoonup}
\nc\dirT{\mathcal T}
\nc\dirF{\mathcal F}
\nc\link{\mathit{link}}
\nc\cT{\mathcal T}

\newcommand{\ssupp}{\mathit{ss}}
\newcommand{\cyl}{\mathit{Cyl}}
\newcommand{\ball}{\mathit{B(x)}}

 \nc\ssf{\mathsf f}
 \nc\ssg{\mathsf g}
\nc\sq{\mathsf q}
 \nc\sQ{\mathsf Q}
 \nc\sR{\mathsf R}

\nc\fa{\mathfrak a}
\nc\fA{\mathfrak A}

\nc\trunc{\mathit{tr}}
\nc\pre{\mathit{pre}}
\nc\expand{\mathit{exp}}

\nc\Sol{\mathit{Sol}}
\nc\direct{\mathit{dir}}

\nc\out{\mathit{out}}
\nc\Morse{\mathit{Morse}}
\nc\arb{\mathit{arb}}
\nc\prearb{\mathit{pre}}

\nc\BZ{\mathbb Z}
\nc\BN{\mathbb N}
\nc\proper{\mathit{prop}}
\nc\torsion{\mathit{tors}}
\nc\Perv{\mathit{Perv}}
\nc\IC{\operatorname{IC}}

\nc\Conv{\operatorname{Conv}}
\nc\Span{\operatorname{Span}}

\nc\image{\operatorname{image}}
\nc\lin{\mathit{lin}}

\nc\inverse{\operatorname{inv}}

\nc\real{\operatorname{Re}}
\nc\imag{\operatorname{Im}}

\nc\fw{\mathfrak w}

\nc\sign{\on{sgn}}
\nc\conic{\mathit{con}}

\nc\even{\mathit{ev}}
\nc\odd{\mathit{odd}}
\nc\add{\mathit{add}}

\nc\orient{\mathit{or}}

\nc\op{\mathit{op}}

\nc\beq{\begin{equation}}
\nc\eeq{\end{equation}}

\nc\bequn{\begin{equation*}}
\nc\eequn{\end{equation*}}

\nc\Cl{\mathit{C\ell}}
\nc\Chek{\mathit{Ch}}
\nc\vanish{\mathit{van}}

\nc\lar{\mathit{large}}
\nc\sma{\mathit{small}}
\nc\rest{\text{rest}}

\nc\cl{\curvearrowright}
\nc\ccl{\curvearrowleft}

\def\clrot{{\rotatebox[origin=c]{180}{$\cl$}}}
\def\cclrot{{\rotatebox[origin=c]{180}{$\ccl$}}}

\nc\mon{\mathit{mon}}
\nc\triv{\mathit{triv}}
\nc\invar{\mathit{\spadesuit}}

\nc{\canfromid}{\mathfrak{p}}
\nc{\cantoid}{\mathfrak{q}}
\nc{\adjmap}{\mathfrak{q}}
\nc{\Cyl}{\text{Cyl}}
\nc{\frakm}{\mathfrak{m}}
\nc{\crit}{\mathit{crit}}

\nc\sh{\sharp}
\nc\frF{\mathfrak F}

\title[Wall-crossing for toric mutations]{Wall-crossing for toric mutations}

\author{David Nadler}
\address{Department of Mathematics\\University of California, Berkeley\\Berkeley, CA  94720-3840}
\email{nadler@math.berkeley.edu}

\begin{abstract}
 This note explains how to deduce the  wall-crossing formula  for toric mutations established by Pascaleff-Tonkonog from the perverse schober of the corresponding local Landau-Ginzburg model.
 Along the way, we develop a general framework to extract a wall-crossing formula from a perverse schober  on the projective line with a single critical point.
 \end{abstract}

\maketitle


\tableofcontents


\section{Introduction}

Our aim in this note is 
 to deduce the wall-crossing formula   for toric mutations established by Pascaleff-Tonkonog~\cite{PT}
 from the perverse schober of the corresponding local Landau-Ginzburg model calculated in~\cite{Nlgps} (see also \cite{N3d} for a concrete approach to the three-dimensional case).
Along the way, we develop a general framework (see Theorem~\ref{thm:intro}) to extract a wall-crossing formula from a perverse schober $\cP$ on the projective line $\BC\BP^1$ with a single critical point $0\in \BC\BP^1$.  
The wall-crossing formula applies to the moduli of objects of the inertia category of  the clean objects  (see Definition~\ref{def:clean}) in the generic fiber  of $\cP$.

\subsection{Wall-crossing formula}\label{s:wallcross}

The  wall-crossing formula for toric mutations established by Pascaleff-Tonkonog~\cite{PT} 
 is the birational map 
\begin{align}
y_i & \mapsto   y_i, \hspace{1em} i=1, \ldots, n-1\label{eq:mut}\\
y_n & \mapsto  y_n^{-1}(1 + y_1 + \cdots + y_{n-1})\nonumber
\end{align}
associated to the mutation $L\rightsquigarrow \mu_D(L)$ of an $n$-dimensional Lagrangian torus $L\simeq (S^1)^n$ around
a singular thimble $D = \Cone(T^\vanish)$ attached to a vanishing $(n-1)$-dimensional  torus $ T^\vanish \simeq (S^1)^{n-1}\subset L$.

The variables $y_i$, $i=1,\ldots, n$, are  simultaneously coordinates on the moduli $\BT^\vee\simeq (\Gm)^n$ of rank one local systems on the Lagrangian torus $L\simeq (S^1)^{n}$ as well as  its mutation $ \mu_D(L)\simeq (S^1)^{n}$ under a natural identification. In particular, 
the variables $y_i$, $i=1,\ldots, n-1$, are  simultaneously coordinates on the moduli $\BT^\vee_0\simeq (\Gm)^{n-1}$ of rank one local systems on the vanishing torus $T^\vanish\simeq (S^1)^{n-1}$ as well as  its mutation $ \mu_D(T^\vanish)\simeq (S^1)^{n-1}$.
%

When $n=2$, the vanishing torus $T^\vanish$ is a circle, the thimble $D = \Cone(T^\vanish)$ is a smooth
disk, and the geometry is that of  a  traditional ``mutation configuration"~\cite[Def. 4.10]{PT}. In general, the geometry is given by the following local model (adapted from~\cite[Sect. 5.4]{PT} to suit our further considerations).


\subsubsection{Local geometry}\label{s:local geom} 

Let $z$ be a coordinate on $\BC$, $z_1, \ldots, z_n$ coordinates on $\BC^n$, and fix the function  $W:\BC^n\to \BC$, $W=z_1\cdots z_n$.

Fix once and for all $\epsilon>0$, and
introduce the symplectic manifold
$ 
M =  \{W\not =  \epsilon\}  \subset \BC^n.
$

\begin{figure}[h!]
\centering
\includegraphics[scale=0.75, trim={5cm 17.25cm 5cm 4cm},clip]{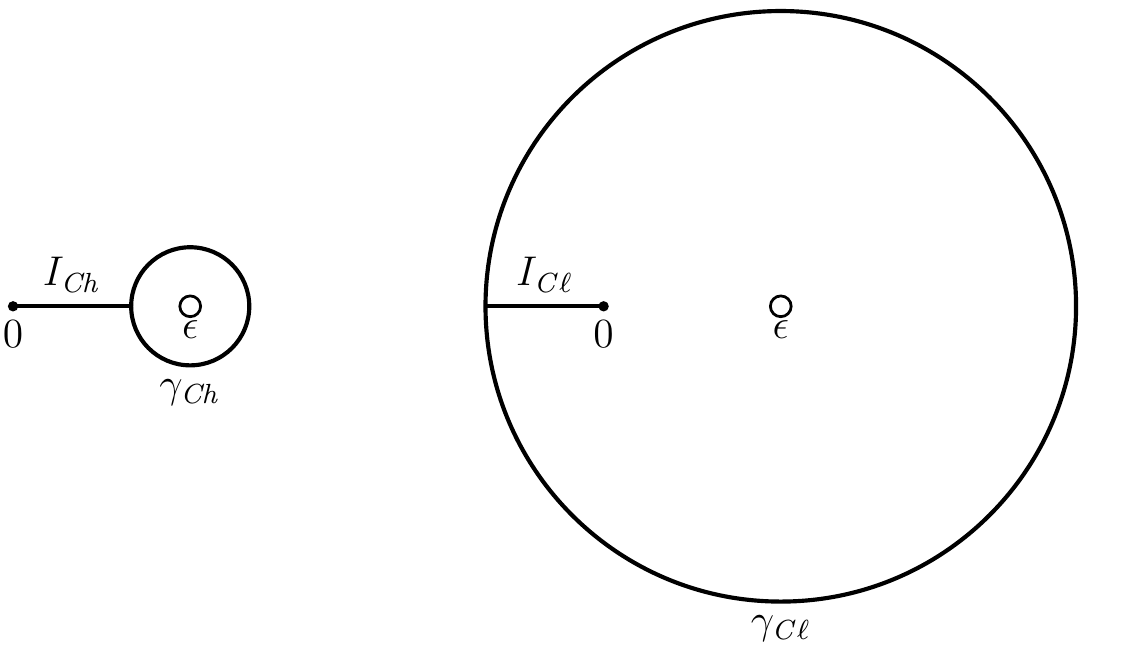}
\caption{Chekanov graph $\Gamma_\Ch$ and  Clifford graph $\Gamma_\Cl$.}
\label{f:skel}
\end{figure}

For a guide to the following constructions, see Figure~\ref{f:skel}.
For any $\rho>0$, $\rho \not = \epsilon$, consider the circle and corresponding Lagrangian torus
\beq 
\xymatrix{
\gamma_\rho = \{|z - \epsilon| = \rho\} \simeq S^1 \subset \BC
&
T_\rho =  \{  W \in \gamma_{\rho}, |z_1| = \cdots = |z_n|\}  \subset M
}\eeq
In particular, for fixed $\rho_\Ch\in (0, \epsilon)$ and $\rho_\Cl\in (\epsilon, \oo)$,  we have the
Chekanov and Clifford circles
and corresponding
 $n$-dimensional Chekanov and Clifford  tori 
 \beq
 \xymatrix{
\gamma_\Ch := \gamma_{\rho_\Ch}
 &
 T_\Ch := T_{\rho_\Ch}
&
 \gamma_\Cl := \gamma_{\rho_\Cl}
 &
 T_\Cl := T_{\rho_\Cl}
}
\eeq

Consider as well the closed intervals 
 \beq
 \xymatrix{
I_\Ch =[0, \epsilon - \rho_\Ch]
 &
I_\Cl = [\epsilon - \rho_\Cl, 0]
}
\eeq
and the corresponding  singular thimbles
\beq
\xymatrix{
D_\Ch =  \{  W \in I_\Ch, |z_1| = \cdots = |z_n|\} \simeq \Cone(T^\vanish_\Ch)}
\eeq
\beq
\xymatrix{
D_\Cl =  \{  W \in I_\Cl, |z_1| = \cdots = |z_n|\}   \simeq  \Cone(T^\vanish_\Cl)
}\eeq
with respective boundaries the vanishing tori
\beq
\xymatrix{
 \partial D_\Ch = T^\vanish_\Ch  =   \{  W =\epsilon-\rho_\Ch, |z_1| = \cdots = |z_n|\}
}
\eeq
\beq
\xymatrix{
 \partial D_\Cl = T^\vanish_\Cl  =   \{  W =\epsilon-\rho_\Cl, |z_1| = \cdots = |z_n|\}
}\eeq

It will also be useful to 
  introduce the Chekanov and Clifford graphs
\beq
\xymatrix{
\Gamma_\Ch = \gamma_{\Ch} \cup I_\Ch
&
\Gamma_\Cl = \gamma_{\large} \cup I_\Cl}
\eeq
which are skeleta of $\BC\setminus \{\epsilon\} \simeq \BC^\times$,
and the  corresponding  Chekanov and Clifford skeleta 
\beq
\xymatrix{
L_\Ch =  \{  W \in \Gamma_\Ch, |z_1| = \cdots = |z_n|\}  
&
L_\Cl =  \{  W \in \Gamma_\Cl, |z_1| = \cdots = |z_n|\}  
}
\eeq
of the symplectic manifold $M =  \{W\not =  \epsilon\}$ itself.
%
%
%

One says the tori $T_\Ch$,
$T_\Cl$ mutate into each other $T_\Cl= \mu_{D_\Ch}(T_\Ch)$, $T_\Ch= \mu_{D_\Cl}(T_\Cl)$ around the respective thimbles $D_\Ch, D_\Cl$ as the radius $\rho$ passes through $\epsilon$ and  the circle $\gamma_\rho$ passes through the critical value $0\in \BC$.

\begin{remark}\label{rem:crit}
We could also work directly with the graph $\Gamma_\crit = \gamma_\epsilon$ given by the circle of critical radius $\epsilon$, and the corresponding critical skeleton
\beq
\xymatrix{
L_\crit =  \{  W \in \Gamma_\crit, |z_1| = \cdots = |z_n|\}  
}
\eeq
It is the union of the two thimbles
\beq
\xymatrix{
D^\pm_\crit =  \{  W \in I_\crit^\pm, |z_1| = \cdots = |z_n|\} \simeq \Cone(T^\vanish_\crit)} 
\eeq
over the two semi-circles $I_\crit^\pm = \{ z\in \Gamma_\crit | \pm \on{Im}(z)\geq 0\}$, with
each thimble the cone over the the same vanishing torus
\beq
\xymatrix{
\partial D^\pm_\crit = T^\vanish_\crit =  \{  W = 2\epsilon,  |z_1| = \cdots = |z_n|\}  
}
\eeq

\begin{figure}[h!]
\centering
\includegraphics[scale=0.75, trim={5cm 18.75cm 11.5cm 4cm},clip]{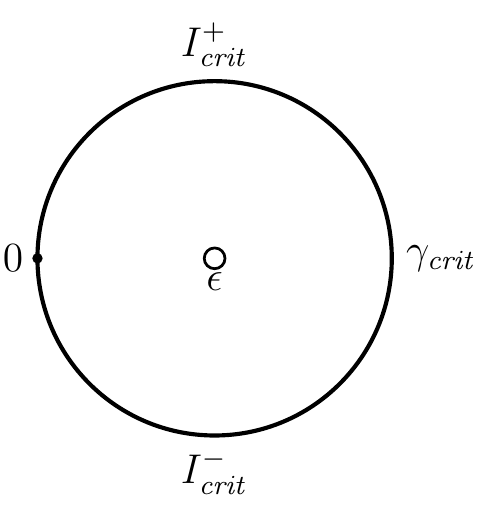}
\caption{Critical graph $\Gamma_\crit$.}
\label{f:critskel}
\end{figure}

\end{remark}

\subsection{Landau-Ginzburg model}\label{s:lgps}

Let us consider the 
Landau-Ginzburg $A$-model of $\BC^n$ with superpotential $W =z_1\cdots z_n$.
By results of~\cite{Nlgps}, its branes are naturally organized by a perverse schober on the disk $D = \{|z| <\epsilon\} \subset \BC\BP^1$ with a single critical point at $0\in D$.
(One expects any Landau-Ginzburg model to similarly give a perverse schober.)

Following Kapranov-Schechtman~\cite{KapS},
a perverse schober  on a disk  $D$ with a single critical point at $0\in D$,
in its minimal one-cut realization, 
 is a spherical functor
\beq
\xymatrix{
\cP_{D, 0}=(S:\cD_\Phi \to  \cD_\Psi)
}
\eeq
from a stable  ``vanishing" dg category $\cD_\Phi$ to  a stable  ``nearby"  dg category $\cD_\Psi$.
(See Sect.~\ref{s: ps} below for more details.)

For the Landau-Ginzburg $A$-model of $\BC^n$ with superpotential $W =z_1\cdots z_n$,
consider the thimble and and cylinder respectively
\beq
\xymatrix{
L =
  \{  W \in \BR_{\geq 0}, |z_1| = \cdots = |z_n|\}  \subset \BC^n
  }
  \eeq
  \beq
  \xymatrix{
L^\circ =
  \{  W \in \BR_{>0} , |z_1| = \cdots = |z_n|\}  \subset \BC^n    
}
\eeq
A main technical result of~\cite{Nlgps} is that we obtain a perverse schober 
\beq\label{eq:aps}
\xymatrix{
\cP_A=(j^*:\mu\Sh_{L}(\BC^n)\ar[r] &  \mu\Sh_{L^\circ}(\BC^n))
}
\eeq
 by taking
 dg categories of microlocal  sheaves
along the  thimble and cylinder,
with their natural restriction  under the open inclusion $j:L^\circ \hra L$. 
 
Next, let us describe the mirror perverse schober. Introduce the torus $\BT^\vee = (\Gm)^n$, with coordinates $x_1, \ldots, x_n$, and its quotient $\BT^\vee_0  = \BT^\vee/\Gm \simeq (\Gm)^{n-1}$ by the diagonal, with coordinates $y_1 = x_1/x_n, \ldots, y_{n-1} = x_1/x_n$.
Consider the inclusion of the pair of pants
\beq
\xymatrix{
i:P_{n-2} = \{1 +  y_1 + \cdots + y_{n-1} = 0\} \ar@{^(->}[r] & \BT^\vee_0
}
\eeq
Then it is elementary to check we  obtain a perverse schober
\beq\label{eq:bps}
\xymatrix{
\cP_B=(i_*:\Perf_\proper(P_{n-2})\ar[r] &  \Perf_\proper(\BT^\vee_0))
}
\eeq
 by taking dg categories of perfect complexes with proper support and the usual pushforward.

A main result of \cite{Nlgps}  is the following mirror equivalence.

\begin{thm}[\cite{Nlgps}]\label{thm:lgps}
There is an equivalence of perverse schobers $\cP_A\risom \cP_B$, i.e. a commutative diagram with vertical equivalences
\beq\label{eq:msps}
\xymatrix{
\ar[d]^-\sim\mu\Sh_{L}(\BC^n)\ar[r]^-{j^*} &  \mu\Sh_{L^\circ}(\BC^n) \ar[d]^-\sim\\
\Perf_\proper(P_{n-2})\ar[r]^-{i_*} &  \Perf_\proper(\BT^\vee_0)
}
\eeq
\end{thm}

\begin{remark}
To pin down the equivalence of the theorem, consider the vanishing torus
\beq
\xymatrix{
T^\vanish =
  \{  W =1, |z_1| = \cdots = |z_n|\}  \subset \BC^n
}
\eeq
and note the
 natural diffeomorphism  $L^\circ \simeq T_\vanish \times \BR_{>0}$.  Note we can canonically identify $\BT^\vee_0$ with the moduli of rank one local systems on $T^\vanish$, and thus  obtain
a  canonical equivalence $\mu\Sh_{L^\circ}(\BC^n) \simeq \Perf_\proper(\BT^\vee_0)$.
This is the right vertical arrow in~\eqref{eq:msps}; it fixes the left vertical arrow, since the functors $j^*, i_*$ are conservative.
 \end{remark}


\subsection{Main result}


We  sketch here  a general wall-crossing formula associated to any perverse schober $\cP_{S^2, 0}$ on the sphere $S^2\simeq \BC\BP^1$ with a single critical point $0\in S^2$. 
We will regard
 such a perverse schober as a pair
\beq
\xymatrix{
\cP_{S^2, 0}=(S:\cD_\Phi \to  \cD_\Psi, \tau:\id \risom T_{\Psi, r})
}
\eeq
 of a spherical functor $S$ and a trivialization $\tau$ of the monodromy functor  $T_{\Psi, r}:\cD_\Psi\risom \cD_\Psi$
 of the nearby category.

Introduce the full subcategory $\cD_\invar\subset \cD_\Psi$ of {\em clean objects} of the nearby category.
We say an object $Y\in \cD_\Psi$ is clean if  we have $S^r(Y) \simeq 0$ (equivalently, $S^\ell(Y) \simeq 0$) for the right adjoint $S^r$ (or left adjoint $S^\ell$) of $S$. 

\begin{remark}
The terminology ``clean" is motivated by the following. If instead of a spherical functor  $S:\cD_\Phi \to  \cD_\Psi$, 
suppose we work with a spherical pair, in the sense of a semi-orthogonal decomposition 
\beq
\xymatrix{
\cD^-_\Phi  \ar[r]^-{I_!} &  \ar@/_1.25em/[l]_-{I^!} \ar@/^1em/[l]^-{I^*}   \cD \ar[r]^-{J^*}& \ar@/_1.25em/[l]_-{J_*} \ar@/^1em/[l]^-{J_!} \cD^+_\Psi
}
\eeq
Then an object $Y\in \cD^+_\Psi$ is clean if and only if 
the canonical map $J_!Y \to J_*Y$ is an isomorphism. 
\end{remark}

Let $\cL\cD_\invar$ denote the {\em inertia dg category} of pairs $(Y, m)$ consisting of $Y\in \cD_\invar$, and an automorphism $m:Y\risom Y$.
For  any automorphism $\mathfrak m$ of the identity 
functor of $\cD_\invar$, we have the natural translated inverse functor
\beq
\xymatrix{
K_\frakm:\cL \cD_\invar  \ar[r]^-\sim & \cL \cD_\invar 
&
K_{\frakm}(Y, m) = (Y, m^{-1}\circ \frakm)
}
\eeq

Finally,  consider the restriction 
\beq
\cP_{\Cyl, 0} = \cP_{S^2, 0}|_\Cyl
\eeq
 of the perverse schober $\cP_{S^2, 0}$ to the cylinder
$\Cyl = \BC \setminus\{\epsilon\} \simeq \BC\BP^1 \setminus \{\epsilon, \oo\}$,
and denote by $\Gamma(\Cyl, \cP_{\Cyl, 0})$  its global sections.

Here is an abstract statement of our main result.

\begin{thm}\label{thm:intro}
 There are natural fully faithful Chekanov and Clifford embeddings fitting into a commutative diagram
 \beq\label{intro: thm diag}
\xymatrix{
\ar[dd]_-{K_\frakm}^-\sim \cL \cD_\invar   \ar@{^(->}[dr]^-{F_\Ch} & \\
& \Gamma(\Cyl, \cP_{\Cyl, 0})\\
\cL \cD_\invar \ar@{^(->}[ru]^-{F_\Cl} &
}
\eeq 
where $\frakm$ is the automorphism of the identity 
\beq
\xymatrix{
\frakm:\id \ar[r]^{\canfromid_\Psi} & T_{\Psi, r} \ar[r]^{\tau^{-1}} & \id
}\eeq
given by the composition of the canonical map $\canfromid_\Psi$ associated to the spherical functor $S$, and the inverse of the trivialization $\tau$
of the nearby monodromy functor.
\end{thm}

\begin{corollary}[Wall-crossing formula]\label{cor:wc}
Under the composition  $F^{-1}_\Cl \circ F_\Ch$, objects of the inertia category $\cL \cD_\invar$ undergo the transformation
\beq\label{eq:intro abstract wc}
\xymatrix{
(Y, m) \ar@{|->}[r] & (Y, m^{-1} \circ \frakm)
}
\eeq
\end{corollary}

\begin{example}\label{ex: intro}
Recall 
the perverse schober
\beq
\xymatrix{
\cP_B=(i_*:\Perf_\proper(P_{n-2})\ar[r] &  \Perf_\proper(\BT^\vee_0))
}
\eeq
 on the disk $D$ with a single critical point $0\in D$.
 
 The clean objects $\Perf_\proper(\BT^\vee_0)_\invar \subset  \Perf_\proper(\BT^\vee_0)$ in the nearby category are those supported away from $P_{n-2}$. Thus  their inertia category has a simple description
 \beq
 \cL\Perf_\proper(\BT^\vee_0)_\invar\simeq \Perf_\proper((\BT^\vee_0 \setminus P_{n-2}) \times \Gm)
 \eeq
 
 The monodromy  
 on the nearby category $ \Perf_\proper(\BT^\vee_0)$ is given by tensoring with $\cO(1)$.
  To smoothly extend $\cP_B$ to a perverse schober $\tilde \cP_B$ on the sphere $S^2$, we must choose a trivialization $\tau:\cO\risom \cO(1)$.
 There are $n$ natural choices given by the homogenous coordinates $x_1, \ldots, x_{n}$. 
 
 Suppose we choose say $x_n$, and write $y_1= x_1/x_n, \ldots, y_{n-1}= x_{n-1}/x_n$. Then the 
 wall-crossing formula  \eqref{eq:intro abstract wc} yields the  birational transformation 
 \begin{align}
y_i & \mapsto   y_i, \hspace{1em} i=1, \ldots, n-1\label{eq:wcf again}\\
y_n & \mapsto  y_n^{-1}(1 + y_1 + \cdots + y_{n-1})\nonumber
\end{align}
originally appearing in \eqref{eq:mut}.

\end{example}
%

\subsubsection{Back to geometry}\label{s:back to geom}
By the mirror equivalence of Theorem~\ref{thm:lgps}, the conclusion of Example~\ref{ex: intro} applies equally well to
the perverse schober
\beq
\xymatrix{
\cP_A=(j^*:\mu\Sh_{L}(\BC^n)\ar[r] &  \mu\Sh_{L^\circ}(\BC^n))
}
\eeq

We explain here the  direct geometric interpretation of the wall-crossing formula \eqref{eq:wcf again}.
We will resume with the constructions and notation introduced in Section~\ref{s:local geom} above. See in particular Figure~\ref{f:skel}.
 
Inside of $\Cyl = \BC \setminus\{\epsilon\} \simeq \BC^\times$,
  recall the Chekanov and Clifford graphs
\beq
\xymatrix{
\Gamma_\Ch = \gamma_{\Ch} \cup I_\Ch
&
\Gamma_\Cl = \gamma_{\large} \cup I_\Cl}
\eeq
and inside of the symplectic manifold  $M =\{W\not = \epsilon\} \simeq (\BC^\times)^n$,
the  corresponding  Chekanov and Clifford skeleta 
\beq
\xymatrix{
L_\Ch =  \{  W \in \Gamma_\Ch, |z_1| = \cdots = |z_n|\}  
&
L_\Cl =  \{  W \in \Gamma_\Cl, |z_1| = \cdots = |z_n|\}  
}
\eeq
From Theorem~\ref{thm:intro}, one can show there are natural equivalences for microlocal sheaves 
\beq
\xymatrix{
 \mu\Sh_{L_\Ch}(M)  \simeq  \Gamma(\Cyl, \cP_A) \simeq  \mu\Sh_{L_\Cl}(M)
}
\eeq

Note we have closed embeddings $T_\Ch \subset L_\Ch$, $T_\Cl \subset L_\Cl$, and hence fully faithful functors
\beq
\xymatrix{
\Loc(T_\Ch) \simeq   \mu\Sh_{T_\Ch}(M) \ar@{^(->}[r] &  \mu\Sh_{L_\Ch}(M)
}
\eeq
\beq
\xymatrix{
\Loc(T_\Cl) \simeq   \mu\Sh_{T_\Cl}(M) \ar@{^(->}[r] &  \mu\Sh_{L_\Cl}(M)
}
\eeq
where we write $\Loc$ for the dg category of local systems.

Since $\gamma_\Ch$ does not wind around $0\in \BC$, we have a canonical splitting
$T^\vanish_\Ch \times \gamma_\Ch\risom T_\Ch.
$
The choice of a coordinate, say  $z_n$, from among $z_1, \ldots, z_n$ provides a parallel splitting
\beq
\xymatrix{
T^\vanish_\Cl \times \gamma_\Cl \ar[r]^-\sim &  T_\Cl 
& 
((z_1, \ldots, z_n), e^{i\theta})  \ar@{|->}[r] & (z_1, \ldots, e^{i\theta} z_k, \ldots, z_n)
}
\eeq
The splittings in turn provide identifications for inertia stacks
\beq
\xymatrix{
 \cL\Loc(T^\vanish_\Ch) \simeq \Loc(T^\vanish_\Ch\times S^1) \simeq 
\Loc(T_\Ch)
}
\eeq
\beq
\xymatrix{
 \cL\Loc(T^\vanish_\Cl) \simeq \Loc(T^\vanish_\Cl\times S^1) \simeq 
\Loc(T_\Cl)
}
\eeq

More fundamentally, the choice of the coordinate $z_n$ provides an extension 
\beq
\xymatrix{
\ol W:\BC^{n-1}\times\BC\BP^1 \ar[r] & \BC\BP^1 & 
 \ol W = [W, z]
}
\eeq 
where we equip $\BC^{n-1}$ with coordinates $z_1, \ldots, z_{n-1}$, and $\BC\BP^1$ with coordinates $[z_n, z]$.
This gives a smooth extension of the perverse schober $\cP_A$ to the sphere $S^2 \simeq \BC\BP^1$.
 
 The circles $\gamma_\Ch, \gamma_\Cl \subset \BC$, with their counterclockwise orientations, are naturally isotopic when regarded within
$\BC\BP^1 \setminus \{0\}\simeq\BC$, but with opposite orientations. 
We also obtain an identification $T^\vanish_\Ch \simeq T^\vanish_\Cl$ by parallel transporting each within $\BC^{n-1}\times\BC\BP^1$  above the respective intervals $[\epsilon - \rho_\Ch, \oo]$, $[-\oo, \epsilon - \rho_\Cl]$
 to the fiber at $\oo$. 

Thus altogether, we have a canonical identification $T_\Ch \simeq T_\Cl$ compatible with the above splittings, but notably with the inverse map on the factor $S^1$.  
Now with the above identifications in hand, we have the following geometric interpretation of Theorem~\ref{thm:intro}.

\begin{thm} For the perverse schober $\cP_A$, the Chekanov and Clifford functors of Theorem~\ref{thm:intro} factor  into the compositions
 \beq
\xymatrix{
F_\Ch:\cL \cD_\invar \ar@{^(->}[r] & \cL\Loc(T^\vanish_\Ch) \simeq
\Loc(T_\Ch) \ar@{^(->}[r] &  \mu\Sh_{L_\Ch}(M)\simeq \Gamma(\Cyl, \cP_A|_{\Cyl})
}
\eeq
 \beq
\xymatrix{
F_\Cl:\cL \cD_\invar \ar@{^(->}[r] & \cL\Loc(T^\vanish_\Cl) \simeq
\Loc(T_\Cl) \ar@{^(->}[r] &  \mu\Sh_{L_\Cl}(M)\simeq \Gamma(\Cyl, \cP_A|_{\Cyl})
}
\eeq
\end{thm}

\begin{corollary}
The wall-crossing formula of Corollary~\ref{cor:wc} is the birational map on moduli of objects for the partially defined functor
$
 \Loc(T_\Ch)\to  \Loc(T_\Cl)
$
 given by 
comparing clean local systems on the Chekanov and Clifford tori as objects in  $\Gamma(\Cyl, \cP_A)$
with coordinates related by the given extension of $\cP_A$ to the sphere $S^2\simeq \BC\BP^1$.
\end{corollary}

\subsection{Acknowledgements}
I thank Vivek Shende for asking whether one could deduce the wall-crossing formula for toric mutations from the structure of the Landau-Ginzburg $A$-model of
 $\BC^n$ with superpotential $W =z_1\cdots z_n$. I thank AIM for hosting the recent workshop~\cite{AIM} at which Vivek  first posed this question. I also thank Dmitry Tonkonog for explaining various aspects of his work.
 
 I am  grateful for the support of NSF grant DMS-1502178.


\section{Perverse schobers}\label{s: ps}
We spell out  constructions with perverse schobers on some simple but useful  complex curves.  


\subsection{On a disk}

Following Kapranov-Schechtman~\cite{KapS}, a perverse schober on the disk $D = \{|z| <\epsilon\} \subset \BC\BP^1$ with a single critical point at $0\in D$,
in  its minimal one-cut realization, is
a {\em spherical functor}
\beq
\xymatrix{
\cP_{D, 0}=(S:\cD_\Phi \to  \cD_\Psi)
}
\eeq
between stable dg categories. By definition, a functor $S$ is spherical if it admits both a left adjoint $S^\ell$ and right adjoint $S^r$, so that we have an adjoint triple $(S^\ell, S, S^r)$
with units and counits of adjunctions denoted by
\beq
\xymatrix{
u_r:\id_\Phi\ar[r] & S^r S & c_r: S S^r \ar[r] & \id_\Psi
}
\eeq
\beq
\xymatrix{
u_\ell:\id_\Psi\ar[r] & S S^\ell & c_\ell: S^\ell S \ar[r] & \id_\Phi 
}
\eeq
Moreover, if we form the natural triangles
\beq
\xymatrix{
T_{\Phi, r} := \Cone(u_r)[-1] \ar[r]^-{\cantoid_\Phi} &  \id_\Phi \ar[r]^-{u_r} & S^r S 
}
\eeq
\beq
\xymatrix{
SS^r \ar[r]^-{c_r} & \id_\Psi \ar[r]^-{\canfromid_\Psi} &  \Cone(u_r) =: T_{\Psi, r} 
}
\eeq
\beq
\xymatrix{
T_{\Psi, \ell} := \Cone(u_\ell)[-1] \ar[r]^-{\cantoid_\Psi} &  \id_\Psi \ar[r]^-{u_\ell} & S S^\ell 
}
\eeq
\beq
\xymatrix{
S^\ell S \ar[r]^-{c_\ell} & \id_\Phi \ar[r]^-{\canfromid_\Phi} &  \Cone(u_\ell) =: T_{\Phi, \ell} 
}
\eeq
then the following properties are required to hold (and in fact, by a theorem of Anno-Logvinenko~\cite{AL}, any two imply  all four):

\begin{enumerate}

\item[(SF1)] $T_{\Psi, r}$ is an equivalence.

\item[(SF2)] The natural composition
\beq
\xymatrix{
S^r \ar[r] & S^r S S^\ell \ar[r] & T_{\Phi, r} S^\ell[1]
}
\eeq
is an equivalence.

 \item[(SF3)] $T_{\Phi, r}$ is an equivalence.

\item[(SF4)] The natural composition
\beq
\xymatrix{
S^\ell  T_{\Psi, r}[-1]  \ar[r] & S^\ell S S^r \ar[r] & S^r
}
\eeq
is an equivalence.
\end{enumerate}

When the above properties hold, 
 $T_{\Phi, \ell}$, $T_{\Psi, \ell}$ are 
 respective inverses of 
  $T_{\Phi, r}$, $T_{\Psi, r}$, and we refer to them as {\em monodromy functors}.

\begin{example}[Smooth hypersurfaces]\label{ex hyper}
Let $X$ be a smooth variety. Let $\cL_X \to  X$ be a line bundle and $\sigma:X\to \cL_X$ a section transverse to the zero section. Let $Y = \{\sigma = 0\}$ be the resulting smooth hypersurface and $i:Y \to X$ its  inclusion.

Let $\Coh(Y)$, $\Coh(X)$  denote the respective dg categories of coherent sheaves.
We will regard the line bundle $\cL_X$ as an object of $\Coh(X)$,
and its restriction $\cL_Y = i^*\cL_X$ as an object of $\Coh(Y)$.
We will regard the section $\sigma$ as a morphism $\sigma:\cO_X\to \cL_X$,
which by duality gives a morphism $\sigma^\vee:\cL_X^\vee\to \cO_X$.

We have the adjoint triple $(i^*, i_*, i^!)$ which satisfies functorial identities
\beq
\xymatrix{
i^*(-) \simeq \cO_Y \otimes_{\cO_X} (-)  
&
i^!(-) \simeq  \cL_Y[-1]\otimes_{\cO_X} (-) 
}
\eeq

Set $\cD_\Phi = \Coh(Y)$,  $\cD_\Psi= \Coh(X)$ and $S= i_*$. Then the monodromy functors satisfy
\beq
\xymatrix{
T_{\Psi, r}(-)  \simeq  \cL_X \otimes_{\cO_X} (-)  &  
T_{\Phi, r}(-)  \simeq  \cL_Y[-2] \otimes_{\cO_Y}(-)
}
\eeq
and hence  both are equivalences. Thus  (SF1) and  (SF3) hold, and so $S=i_*$ is a spherical functor.
We will denote this
perverse schober by 
\beq
\cP_{Y \subset X} = (i_*:\Coh(Y) \to \Coh(X))
\eeq
\end{example}

 
 \subsection{On a sphere}

Suppose given a perverse schober on the disk  $D = \{|z| <\epsilon\} \subset \BC\BP^1$  with a single critical point $0\in D$ in its realization as a spherical functor
\beq
\xymatrix{
\cP_{D, 0}=(S:\cD_\Phi \to  \cD_\Psi)
}
\eeq
Recall we then have  inverse monodromy functors $T_{\Psi, r}$, $T_{\Psi, \ell}$ on the nearby category $\cD_\Psi$.

Let us formulate what it means to extend $\cP_{D, 0}$ smoothly to the sphere $S^2 \simeq \BC\BP^1$.

\begin{defn} (1) A {\em framing} for the perverse schober
$\cP_{D, 0}=(S:\cD_\Phi \to  \cD_\Psi)$ 
is a trivialization of the monodromy of the nearby category
\beq\label{eq:framing}
\xymatrix{
\tau:\id \ar[r]^-\sim &  T_{\Psi, r}
}
\eeq

(2) 
A perverse schober on the  sphere $S^2 \simeq \BC\BP^1$
with a single critical point $0\in S^2$ is  a pair 
\beq
\xymatrix{
\cP_{S^2, 0}=(S:\cD_\Phi \to  \cD_\Psi, \tau:\id \risom T_{\Psi, r})
}
\eeq
consisting of a spherical functor $S$
and  a framing $\tau$.
\end{defn}

\begin{example}[Smooth hypersurfaces]\label{ex hyper frame}
Recall from Example~\ref{ex hyper} the perverse schober
\beq
\cP_{Y \subset X} = (i_*:\Coh(Y) \to \Coh(X))
\eeq
Recall the monodromy of the nearby category is given by
\beq
\xymatrix{
T_{\Psi, r}(-)  \simeq  \cL_X \otimes_{\cO_X} (-) 
}
\eeq
Thus framings are equivalent to  isomorphisms
\beq
\xymatrix{
\tau:\cO_X \ar[r]^-\sim &  \cL_X
}
\eeq
or in other words, non-vanishing sections of $ \cL_X$.
\end{example}

 
 \subsection{On a cylinder}
 Suppose given perverse schober on the disk  $D = \{|z| <\epsilon\} \subset \BC\BP^1$  with a single critical point $0\in D$ in its realization as a spherical functor
 \beq
\xymatrix{
\cP_{D, 0}=(S:\cD_\Phi \to  \cD_\Psi)
}
\eeq

 Let us formulate what it means to extend $\cP_{D, 0}$ smoothly to the cylinder $\Cyl = \BC \setminus\{\epsilon\} \simeq \BC\BP^1 \setminus \{\epsilon, \oo\}$.

\begin{defn}\label{def:pd cyl}
(1) A perverse schober on the  cylinder $\Cyl = \BC \setminus\{\epsilon\} \simeq \BC\BP^1 \setminus \{\epsilon, \oo\}$
with a single critical point $0\in \Cyl$  is a pair 
\beq
\xymatrix{
\cP_{\Cyl, 0}=(S:\cD_\Phi \to  \cD_\Psi, M:\cD_\Psi\risom  \cD_\Psi)
}
\eeq
consisting of a spherical functor $S$ and an additional monodromy  functor 
$M$.

(2) The {\em global sections} $\Gamma(\Cyl, \cP_{\Cyl, 0})$ is the dg category of quintuples $(X,Y_1, Y_2, \Delta, m)$
consisting of $X\in \cD_\Phi$, $Y_1, Y_2\in \cD_\Psi$, an exact triangle 
\beq
\xymatrix{
\Delta =(Y_1[-1] \ar[r]^-i & S(X)\ar[r]^-p & Y_2\ar[r]^-{\partial} & Y_1)
}
\eeq
and an isomorphism 
\beq
\xymatrix{
m:Y_1 \ar[r]^-\sim &  M(Y_2)
}
\eeq
Morphisms are maps of diagrams: a
morphism $(X, Y_1, Y_2, \Delta, m) \to (X', Y'_1, Y'_2, \Delta', m')$ consists of maps $f:X \to X'$, $g_1:Y_1\to Y_1'$, $g_2:Y_2\to Y_2'$, with a lift to a map of triangles 
\beq
\xymatrix{
\ar@<-1.5em>[d]_-g\Delta =  (\ar@<1.25em>[d]^-{g_1[-1]}Y_1[-1]\ar[r]^-i & S(X) \ar[d]^-{S(f)} \ar[r]^-p & 
Y_2\ar[d]^-{g_2} \ar[r]^-{\partial} & Y_1\ar[d]^-{g_1} )\\
\Delta' =  (Y_1'[-1]\ar[r]^-i & S(X')\ar[r]^-p & Y_2'\ar[r]^-{\partial} & Y_1' )
}\eeq
and a commutativity isomorphism $M(g_2) \circ m \simeq m' \circ g_1$.
\end{defn}

\begin{remark}
It is always possible to extend a  perverse schober $\cP_{D, 0}$ on the disk $D$ with a single critical point $0\in D$ to one
on the  cylinder $\Cyl$ by  taking $M$ to be the identity functor. 
\end{remark}

\begin{remark}
It is always possible to restrict a perverse schober $\cP_{S^2, 0}$ on the sphere $S^2$ with a single critical point 
$0\in S^2$ to one
on the  cylinder $\Cyl$ by forgetting the framing $\tau$ and taking $M$ to be the identity functor.
\end{remark}

\begin{remark}
The notion of global sections  $\Gamma(\Cyl, \cP_{\Cyl, 0})$ given in Definition~\ref{def:pd cyl}(2) follows from considering sections supported over the specific Lagrangian skeleton $\Gamma_\Ch = \gamma_\Ch \cup I_\Ch \subset \Cyl$. We will see immediately below an equivalent notion  of global sections that follows from considering
sections supported over the alternative Lagrangian skeleton $\Gamma_\Cl = \gamma_\Cl \cup I_\Cl \subset \Cyl$.
 \end{remark}
%
%

\subsubsection{Mutated global sections} 
The notion of  global sections $\Gamma(\Cyl, \cP_{\Cyl, 0})$ given above in Definition~\ref{def:pd cyl}(2) is one of many equivalent possibilities. We introduce here a key alternative notion and show it is equivalent to the original.

\begin{defn}\label{def:mut gs}
The {\em mutated global sections} $\Gamma^\sh(\Cyl, \cP_{\Cyl, 0})$ is the dg category of quintuples $(X^\sh,Y^\sh_1, Y^\sh_2, \Delta^\sh, m^\sh)$
consisting of $X^\sh\in \cD_\Phi$, $Y^\sh_1, Y^\sh_2\in \cD_\Psi$, an exact triangle 
\beq\label{eq: given tri}
\xymatrix{
\Delta^\sh =(T_{\Psi, r}(Y^\sh_1)[-1] \ar[r]^-i & S(X^\sh)\ar[r]^-p & Y^\sh_2\ar[r]^-{\partial} & T_{\Psi, r}(Y^\sh_1))
}
\eeq
and an isomorphism 
\beq
\xymatrix{
m^\sh: M(Y^\sh_1) \ar[r]^-\sim &  Y^\sh_2
}
\eeq
Morphisms are maps of diagrams as in Definition~\ref{def:pd cyl}(2).
\end{defn}

Now we will define a mutation equivalence
\begin{equation*}
\xymatrix{
\frF^\sh:\Gamma(\Cyl, \cP_{\Cyl, 0}) \ar[r] & \Gamma^\sh(\Cyl, \cP_{\Cyl, 0}) &
\frF^\sh(X, Y_1, Y_2, \Delta, m) = (X^\sh, Y^\sh_1, Y^\sh_2, \Delta^\sh, m^\sh) 
}
\end{equation*}

First, set $Y_1^\sh := Y_2$, $Y_2^\sh := Y_1$, and
\beq
\xymatrix{
m^\sh:M(Y^\sh_1)  =  M(Y_2) \ar[r]^-{m^{-1}}  & Y_1 = Y_2^\sh
}
\eeq

By adjunction, the map $p:S(X)\to Y_2$ provides a map $\tilde p: X\to S^r(Y_2)$, and we set $X^\sh := \Cone(\tilde p)$.
To construct the triangle $\Delta^\sh$, 
note that adjunction  further provides a commutative diagram
\beq\label{eq:commadj}
\xymatrix{
\ar[dr]_-{p} S(X)\ar[r]^-{S(\tilde p)} &  S(S^r(Y_2)) \ar[d]^-{c_r} \\
&  Y_2
}
\eeq
Recall as well the natural triangle
\beq
\xymatrix{
 SS^r \ar[r]^-{c_r} & \id  \ar[r]^-{\canfromid_\Psi} & T_{\Psi, r} \ar[r] & SS^r[1]
}
\eeq
  Taking the cone  of each map  of \eqref{eq:commadj}, we obtain another  triangle
  \beq
\xymatrix{
 \Cone(S(\tilde p)) \ar[r] & Y_1 \ar[r]^-{\tilde \canfromid_\Psi} & T_{\Psi, r}(Y_2) \ar[r] & \Cone(S(\tilde p))[1]
}
\eeq
Now we take $\Delta^\sh$ to be the rotated triangle
\beq\label{eq:key rot tri}
\xymatrix{
\Delta^\sh = (T_{\Psi, r}(Y_2)[-1] \ar[r] & \Cone(S(\tilde p)) \ar[r] & Y_1 \ar[r]^-{\tilde \canfromid_\Psi} & T_{\Psi, r}(Y_2))
}
\eeq
using the canonical identification $\Cone(S(\tilde p)) \simeq  S(\Cone(\tilde p))$.

\begin{remark}\label{key rem}
Let us highlight a key property of the map $\tilde \canfromid_\Psi:Y_1 \to T_{\Psi, r}(Y_2)$ appearing in the triangle $\Delta^\sh$. By construction, it arises in a commutative diagram
\beq\label{eq:commadj}
\xymatrix{
\ar[dr]_-{\canfromid_\Psi} Y_2 \ar[r]^-{\partial} &  Y_1 \ar[d]^-{\tilde \canfromid_\Psi}\\
&  T_{\Psi, r}(Y_2)
}
\eeq
where $\partial$ is given in the initial triangle~\eqref{eq: given tri}, and $\canfromid_\Psi$ is the canonical map.
In particular, if we have $Y_1 = Y_2$, and the map $\partial:Y_2\to Y_1$ is the identity, then $\tilde \canfromid_\Psi =  \canfromid_\Psi$.
\end{remark}
%
%

\begin{lemma}
The mutation functor
\begin{equation*}
\xymatrix{
\frF:\Gamma(\Cyl, \cP_\Cyl) \ar[r] & \Gamma^\sh(\Cyl, \cP_\Cyl) &
\frF(X, Y_1, Y_2, \Delta, m) = (X^\sh, Y^\sh_1, Y^\sh_2, \Delta^\sh, m^\sh) 
}
\end{equation*}
 is an equivalence.
 \end{lemma}

\begin{proof}
We can define an inverse to $\frF$ beginning with the evident assignments 
 $Y_1 := Y_2^\sh$, $Y_1 := Y_2^\sh$,
\beq
\xymatrix{
m:Y_1  =  Y_2^\sh \ar[r]^-{m^{-1}}  & M(Y_1^\sh) = Y_2
}
\eeq

For $X$ and the triangle $\Delta$, we proceed  as follows. 
By adjunction, the map $i:T_{\Psi, r}(Y_1^\sh)[-1] \to S(X^\sh)$ provides a map $\tilde \imath: S^\ell(T_{\Psi, r}(Y_1^\sh))[-1] \to X^\sh$, and we take  $X := \Cone(\tilde \imath)[-1]$. 

To construct the triangle $\Delta$, 
note that adjunction  further provides a commutative diagram
\beq\label{eq:commadjsh}
\xymatrix{
SS^\ell(T_{\Psi, r}(Y_1^\sh))[-1] \ar[r]^-{S(\tilde \imath)} &  S(X^\sh)\\
T_{\Psi, r}(Y_1^\sh)[-1] \ar[u]^-{u_\ell} \ar[ur]_-i & 
}
\eeq
Recall (SF4) confirms the natural composition is an equivalence. 
\beq
\xymatrix{
S^\ell  T_{\Psi, r}[-1]  \ar[r] & S^\ell S S^r \ar[r] & S^r
}
\eeq
and thus we can write \eqref{eq:commadjsh} in the form
\beq\label{eq:commadjsh2}
\xymatrix{
SS^r (Y_1^\sh) \ar[r]^-{S(\tilde \imath)} &  S(X^\sh)\\
T_{\Psi, r}(Y_1^\sh)[-1] \ar[u] \ar[ur]_-i & 
}
\eeq
  Taking the cone  of each map  of \eqref{eq:commadjsh2}, we obtain another  triangle
  \beq
\xymatrix{
Y_1^\sh \ar[r] & Y_2^\sh \ar[r] &  \Cone(S(\tilde \imath)) \ar[r] & Y_1^\sh[-1]
}
\eeq
We take $\Delta$ to be the rotated shifted triangle
\beq
\xymatrix{
\Delta = (Y_2^\sh[-1] \ar[r] &  \Cone(S(\tilde \imath))[-1] \ar[r] & Y_1^\sh \ar[r] & Y_2^\sh)
}
\eeq
using the canonical identification $\Cone(S(\tilde \imath)) \simeq  S(\Cone(\tilde \imath))$.

We leave it to the reader to check the constructed functor is inverse to $\frF$.
\end{proof}


\section{Abstract wall-crossing}\label{s: awc}
In this section, we will explain  how a perverse schober on the sphere $S^2\simeq \BC\BP^1$ with a single critical point 
provides a wall-crossing formula. 

 
\subsection{Chekanov functor}

 Suppose given a perverse schober
\beq
\xymatrix{
\cP_{\Cyl, 0}=(S:\cD_\Phi \to \cD_\Psi, M:\cD_\Psi\risom \cD_\Psi)
}
\eeq
 on the  cylinder $\Cyl = \BC \setminus\{\epsilon\} \simeq \BC\BP^1 \setminus \{\epsilon, \oo\}$
with a single critical point $0\in \Cyl$.
Thus by definition, $S$ is a spherical functor, and $M$ is an invertible functor.

Following Definition~\ref{def:pd cyl}(2), the global sections $\Gamma(\Cyl, \cP_{\Cyl, 0})$ is the dg category of quintuples $(X,Y_1, Y_2, \Delta, m)$
consisting of $X\in \cD_\Phi$, $Y_1, Y_2\in \cD_\Psi$, an exact triangle 
\beq
\xymatrix{
\Delta =(Y_1[-1] \ar[r]^-i & S(X)\ar[r]^-p & Y_2\ar[r]^-{\partial} & Y_1)
}
\eeq
and an isomorphism 
$
m:Y_1 \risom   M(Y_2)
$

 Introduce the smooth perverse schober 
\beq
 \cP_\sm = (0:\{0\} \to  \cD_\Psi,M: \cD_\Psi\risom \cD_\Psi)
\eeq
on the  cylinder 
with no critical points. 

Following Definition~\ref{def:pd cyl}(2), the global sections $\Gamma(\Cyl, \cP_\sm)$ is the dg category of quintuples $(0,Y_1, Y_2, \Delta, m)$
consisting of $Y_1, Y_2\in \cD_\Psi$, an exact triangle 
\beq\label{eq:spl tri}
\xymatrix{
\Delta =(Y_1[-1] \ar[r]^-i & 0\ar[r]^-p & Y_2\ar[r]^-{\partial} & Y_1)
}
\eeq
and an isomorphism 
$m:Y_1 \risom M(Y_2).
$

Thus  $\Gamma(\Cyl, \cP_\sm) \subset \Gamma(\Cyl, \cP_\Cyl)$  is the full subcategory
of quintuples $(0, Y_1, Y_2, \Delta, m)$ with vanishing first entry.

\begin{remark}
There is an evident commutative  diagram
\beq\label{eq: ev ps map}
\xymatrix{
\ar[d]_-0\{0\} \ar[r]^-0 &  \cD_\Psi  \ar[d]^-\id  \\
\cD_\Phi \ar[r]^-S &  \cD_\Psi\\
}
\eeq
compatible with $M:\cD_\Psi\risom \cD_\Psi$ that induces the inclusion 
$\Gamma(\Cyl, \cP_0) \subset \Gamma(\Cyl, \cP_\Cyl)$.
In general, passing to the adjoints to the horizontal maps in \eqref{eq: ev ps map} does not lead to a commutative diagram. Thus it is not clear whether to consider \eqref{eq: ev ps map} as a map of perverse schobers $\cP_\sm \to \cP_\Cyl$.
But we will only consider \eqref{eq: ev ps map} where in the top right we take the full subcategory $\cD_\invar\subset \cD_\Psi$ of clean objects (see Definition~\ref{def:clean}) for which both adjoints to $S$ vanish, and so  indeed fit into a trivially commutative diagram.
\end{remark}

Let $\cL^M \cD_\Psi$ denote the dg path category with objects pairs $(Y, m)$ consisting of $Y\in \cD_\Psi$, 
and a monodromy isomorphism 
$m:Y \risom M(Y)$;
morphisms $(Y, m)\to (Y', m')$ are maps $g:Y\to Y'$ with an isomorphism
$M(g) \circ m \simeq m' \circ g$.

\begin{remark}
When $M=\id$, note that $\cL^M \cD_\Psi $ is the usual inertia dg category $\cL\cD_\Psi$
of pairs $(Y, m)$ 
consisting of $Y\in \cD_\Psi$, 
and a monodromy isomorphism 
$m:Y \risom Y$.
\end{remark}

Introduce the  {\em Chekanov functor} 
\beq
\xymatrix{
F_\Ch:\cL^M \cD_\Psi \ar[r] & \Gamma(\Cyl, \cP_\Cyl) 
&
F_\Ch(Y, m) = (0,Y, Y, \Delta_0, m) 
}
\eeq
where $\Delta_0$ is the split triangle
\beq
\xymatrix{
\Delta_0 =(Y[-1]\ar[r] & 0\ar[r] & Y\ar[r]^-{\id} & Y)
}\eeq

\begin{lemma}\label{l:equiv ch}
$F_\Ch$ is fully faithful with essential image $\Gamma(\Cyl, \cP_\sm) \subset \Gamma(\Cyl, \cP_{\Cyl, 0})$.
\end{lemma}

\begin{proof}
We may use the isomorphism $\partial:Y_2\risom Y_1$ in the triangle \eqref{eq:spl tri sh} to identify $Y_1, Y_2$.
Thus any quintuple $(0, Y_1, Y_2, \Delta, m)$  with vanishing first entry is in the essential image of $F_\Ch$.

Recall morphisms $(0, Y, Y, \Delta, m) \to (0, Y', Y', \Delta', m')$ are maps $g_1:Y\to Y'$, $g_2:Y\to Y'$, with a lift to a map of triangles 
\beq
\xymatrix{
\ar@<-1.5em>[d]_-g\Delta =  (\ar@<1.25em>[d]^-{g_1[-1]}Y[-1]\ar[r] & 0 \ar[d] \ar[r] & 
Y\ar[d]^-{g_2} \ar[r]^-{\id} & Y\ar[d]^-{g_1} )\\
\Delta' =  (Y'[-1]\ar[r] & 0\ar[r]^-p & Y'\ar[r]^-{\id} & Y' )
}\eeq
and a commutativity isomorphism $M(g_2) \circ m \simeq m' \circ g_1$. The lift to triangles provides an isomorphism $g_1\simeq g_2$, and so morphisms are simply maps $g:Y\to Y'$ such that $M(g) \circ m \simeq m' \circ g$ as in the  dg path category.
\end{proof}


\subsection{Clifford functor}
We repeat here the constructions of the preceding section in a parallel mutated form.

 Consider again the given  perverse schober
\beq
\xymatrix{
\cP_{\Cyl, 0}=(S:\cD_\Phi \to \cD_\Psi, M:\cD_\Psi\risom \cD_\Psi)
}
\eeq
 on the  cylinder $\Cyl = \BC \setminus\{\epsilon\} \simeq \BC\BP^1 \setminus \{\epsilon, \oo\}$
with a single critical point $0\in \Cyl$.

Following Definition~\ref{def:mut gs}, the mutated global sections $\Gamma^\sh(\Cyl, \cP_{\Cyl, 0})$ is the dg category of quintuples $(X^\sh,Y^\sh_1, Y^\sh_2, \Delta^\sh, m^\sh)$
consisting of $X^\sh\in \cD_\Phi$, $Y^\sh_1, Y^\sh_2\in \cD_\Psi$, an exact triangle 
\beq
\xymatrix{
\Delta_0^\sh =(T_{\Psi, r}(Y^\sh_1)[-1] \ar[r]^-i & S(X^\sh)\ar[r]^-p & Y^\sh_2\ar[r]^-{\partial} & T_{\Psi, r}(Y^\sh_1))
}
\eeq
and an isomorphism 
$m^\sh: M(Y^\sh_1) \risom   Y^\sh_2$

 Consider again the smooth perverse schober 
\beq
 \cP_\sm = (0:\{0\} \to  \cD_\Psi,M: \cD_\Psi\risom \cD_\Psi)
\eeq
on the  cylinder 
with no critical points. 

Following Definition~\ref{def:mut gs}, the mutated global sections $\Gamma^\sh(\Cyl, \cP_\sm)$ is the dg category of quintuples $(0,Y^\sh_1, Y^\sh_2, \Delta_0^\sh, m^\sh)$
consisting of $Y^\sh_1, Y^\sh_2\in \cD_\Psi$, an exact triangle 
\beq\label{eq:spl tri sh}
\xymatrix{
\Delta^\sh_0 =(T_{\Psi, r}(Y^\sh_1)[-1] \ar[r]^-i & 0\ar[r]^-p & Y^\sh_2\ar[r]^-{\partial} & T_{\Psi, r}(Y^\sh_1))
}
\eeq
and an isomorphism 
$m^\sh: M(Y^\sh_1) \risom  Y^\sh_2$.
%

Thus  $\Gamma^\sh(\Cyl, \cP_\sm) \subset \Gamma^\sh(\Cyl, \cP_\Cyl)$  is the full subcategory
of quintuples $(0, Y^\sh_1, Y^\sh_2, \Delta^\sh, m^\sh)$ with vanishing first entry.

\begin{remark}
The inclusion 
$\Gamma^\sh(\Cyl, \cP_\sm) \subset \Gamma^\sh(\Cyl, \cP_{\Cyl, 0})$ is induced
by the diagram \eqref{eq: ev ps map}.
\end{remark}

Let $\cL_{M\circ T_{\Psi, r}} \cD_\Psi$ denote the dg path category of pairs $(Y^\sh, m^\sh)$ consisting of $Y^\sh\in \cD_\Psi$, 
and a monodromy isomorphism 
$m^\sh:M(T_{\Psi, r}(Y^\sh)) \risom Y^\sh$;
morphisms $(Y^\sh, m^\sh)\to ((Y^\sh)', (m^\sh)')$ are maps $g:Y\to (Y^\sh)'$ with an isomorphism
$g \circ m^\sh \simeq (m^\sh)' \circ M(T_{\Psi,r} (g))$.

\begin{remark}
When $M\circ T_{\Psi, r}=\id$, note that $\cL_{M\circ T_{\Psi, r}} \cD_\Psi $ is the usual inertia dg category $\cL\cD_\Psi$.
%
\end{remark}

%
%



%

Introduce the  {\em Clifford functor} 
\beq
\xymatrix{
F_\Cl:\cL_{M\circ T_{\Psi, r}} \cD_\Psi \ar[r] & \Gamma^\sh(\Cyl, \cP_{\Cyl, 0}) 
&
F_\Cl(Y^\sh, m^\sh) = (0,T_{\Psi, r}^{-1}(Y^\sh), Y^\sh, \Delta_0^\sh, m^\sh) 
}
\eeq
where $\Delta^\sh_0$ is the split triangle
\beq
\xymatrix{
\Delta^\sh_0 =(Y^\sh[-1]\ar[r] & 0\ar[r] & Y^\sh\ar[r]^-{\id} & Y^\sh)
}\eeq

\begin{lemma}
$F_\Cl$ is fully faithful with essential image $\Gamma^\sh(\Cyl, \cP_\Cl) \subset \Gamma^\sh(\Cyl, \cP_\Cyl)$.
\end{lemma}

\begin{proof}

We may use the isomorphism $\partial:Y^\sh_2\risom T_{\Psi, r}(Y^\sh_1)$ in triangle~\eqref{eq:spl tri sh} to identify $T_{\Psi, r}(Y^\sh_1), Y^\sh_2$. Thus any quintuple $(0, Y^\sh_1, Y^\sh_2, \Delta^\sh, m^\sh)$  with vanishing first entry is in the essential image of $F_\Cl$.

The rest of the proof is the
same as that of Lemma~\ref{l:equiv ch}.
\end{proof}


\subsection{Framing identifications}

Now suppose given a perverse schober 
\beq
\cP_{S^2, 0} = (S:\cD_\Phi \to  \cD_\Psi, \tau:\id \risom  T_{\Psi, r})
\eeq
 on the  sphere $S^2 \simeq \BC\BP^1$
with a single critical point at $0\in S^2$.

Let us restrict  to a perverse schober 
\beq
\cP_{\Cyl, 0} =  \cP_{S^2, 0}|_\Cyl=(S:\cD_\Phi \to  \cD_\Psi, \id: \cD_\Psi\risom \cD_\Psi)
\eeq
 on the  cylinder $\Cyl = \BC \setminus\{\epsilon\} \simeq \BC\BP^1 \setminus \{\epsilon, \oo\}$
by
 forgetting the framing $\tau$ and taking the additional monodromy functor $M$ to be the identity.
 Consider the smooth perverse schober
\beq
 \cP_\sm = (0:\{0\} \to  \cD_\Psi,\id: \cD_\Psi\risom \cD_\Psi)
\eeq
again taking the additional monodromy functor $M$ to be the identity.

Observe that the framing $\tau:\id \risom  T_{\Psi, r}$ provides a canonical equivalence 
$\tau:\cL\cD_\Psi \risom  \cL_{T_{\Psi, r}}  \cD_\Psi$.
Thus the Chekanov and Clifford functors take the respective forms
\beq
\xymatrix{
F_\Ch:\cL \cD_\Psi \ar[r] & \Gamma(\Cyl, \cP_\Cyl) 
&
F_\Ch(Y, m) = (0,Y, Y, \Delta_0, m) 
}
\eeq
\beq
\xymatrix{
\Delta_0 =(Y[-1]\ar[r] & 0\ar[r] & Y\ar[r]^-{\id} & Y)
}\eeq
\beq
\xymatrix{
F_\Cl:\cL \cD_\Psi \ar[r] & \Gamma^\sh(\Cyl, \cP_\Cyl) 
&
F_\Cl(Y^\sh, m^\sh) = (0,Y^\sh, Y^\sh, \Delta_0^\sh, m^\sh) 
}
\eeq
\beq
\xymatrix{
\Delta^\sh_0 =(Y^\sh[-1]\ar[r] & 0\ar[r] & Y^\sh\ar[r]^-{\id} & Y^\sh)
}\eeq

\begin{remark}
 Recall  the mutation equivalence $\frF^\sh:\Gamma(\Cyl, \cP_{\Cyl, 0})\risom  \Gamma^\sh(\Cyl, \cP_{\Cyl, 0})$.
 We caution the reader that the following diagram is {\em not} commutative
\beq
\xymatrix{
 & \Gamma(\Cyl, \cP_{\Cyl, 0})\ar[dd]^-{\frF^\sh}\\
\cL \cD_\Psi\ar[ur]^-{F_\Ch} \ar[dr]^-{F_\Cl} &  \\
&  \Gamma^\sh(\Cyl, \cP_{\Cyl, 0})
}
\eeq 
We will introduce a correct intertwining relation in the next section.
\end{remark}


\subsection{Clean objects}

 Let us continue with the setup of the preceding section.

\begin{defn}\label{def:clean}
We say an object $Y \in \cD_\Psi$ is {\em clean} if $S^r(Y) \simeq 0$.

We denote by $\cD_\invar \subset \cD_\Psi$ the full subcategory of clean objects.
\end{defn}

\begin{lemma}
An object $Y \in \cD_\Psi$ is {clean} if and only if any of the following hold:
\begin{enumerate}
\item $S^\ell(Y) \simeq 0$.
\item
 $\canfromid_\Psi: \id_\Psi \to T_{\Psi, r}$ 
 is an isomorphism.
\item $\cantoid_\Psi:T_{\Psi, \ell} \to  \id_\Psi$ 
 is an isomorphism.
\end{enumerate}
\end{lemma}

\begin{proof}
For the equivalence with property (1), recall (SF2):
the natural composition
$$
\xymatrix{
S^r \ar[r] & S^r S S^\ell \ar[r] & T_{\Phi, r} S^\ell[1]
}
$$
is an equivalence and (SF3): $T_{\Phi, r}$ is an equivalence.

The equivalences with properties (2) and (3) are immediate from the triangles defining the canonical maps.
\end{proof}

\begin{remark}
Note that the monodromy $T_{\Psi, r}$, and hence also its inverse $T_{\Psi, \ell}$,  preserves clean objects
by property (1) of the lemma and (SF4).
\end{remark}

\begin{remark}
It is important to distinguish between a framing 
$
\tau:\id \risom T_{\Psi, r}
$
 and the canonical map
$
\canfromid_\Psi: \id_\Psi \to  T_{\Psi, r}.
$
The framing $\tau$ is an additional structure not intrinsic to the perverse schober, and by definition, it is required to be an isomorphism. The canonical  map $\canfromid_\Psi$
is intrinsic to the perverse schober, but not necessarily an isomorphism when evaluated on some objects. 
\end{remark}

\begin{example}[Smooth hypersurfaces]\label{ex hyper clean}
Recall from Example~\ref{ex hyper} the perverse schober
\beq
\cP_{Y \subset X} = (i_*:\Coh(Y) \to \Coh(X))
\eeq
An object $\cF\in \Coh(X)$ is clean
if and only if $i^*\cF \simeq 0$. Equivalently,
an object $\cF\in \Coh(X)$ is clean
if and only if either (hence both) of the maps 
 \beq
 \xymatrix{
 \cF\ar[r]^-\sigma &  \cF \otimes_{\cO_X} \cL_X
&  \cF \otimes_{\cO_X} \cL^\vee_X \ar[r]^-{\sigma^\vee} &   \id
 }
 \eeq
  is an isomorphism. From any of the above descriptions, we see that 
  an object $\cF\in \Coh(X)$ is clean if and only if it is supported away from $Y = \{\sigma = 0\}$.
\end{example}

Now introduce the perverse schober 
\beq
 \cP_\invar= (0:\{0\} \to  \cD_\invar,\id: \cD_\invar\risom \cD_\invar)
\eeq
on the  cylinder 
with no critical points by taking the full dg subcategory  $\cD_\invar\subset \cD_\Psi$ of clean objects
and the additional monodromy functor $M$ to be the identity.

Given any automorphism $\mathfrak m$ of the identity 
functor of $\cD_\invar$, we have a corresponding translated inverse functor on the dg inertia category
\beq
\xymatrix{
K_\frakm:\cL \cD_\invar  \ar[r]^-\sim & \cL \cD_\invar 
&
K_{\frakm}(Y, m) = (Y, m^{-1}\circ \frakm)
}
\eeq
Note we have a canonical identification $m^{-1}\circ \frakm\simeq \frakm \circ m^{-1}$ by functoriality.

Now we have the intertwining:

\begin{prop}\label{p:intertwine}
 The restrictions of the  Chekanov and Clifford functors to clean objects fits into a canonically commutative diagram
 \beq
\xymatrix{
\ar[d]_-{K_\frakm} \cL \cD_\invar \ar[r]^-{F_\Ch} & \Gamma(\Cyl, \cP_\Cyl)\ar[d]^-{\frF^\sh}\\
\cL \cD_\invar \ar[r]^-{F_\Cl} &  \Gamma^\sh(\Cyl, \cP_\Cyl)
}
\eeq 
where the automorphism $\frakm$  of the identity 
functor of $\cD_\invar$ is the composition 
\beq
\xymatrix{
\frakm : \id \ar[r]^-{\canfromid_\Psi} &   T_{\Psi, r} \ar[r]^-{\tau^{-1}} & \id
}
\eeq
of the  canonical morphism $\canfromid_\Psi$ (which is invertible on clean objects)
and the inverse of the framing $\tau$.
\end{prop}

\begin{proof} 
For $(Y, m) \in \cL \cD_\invar$, we can apply the definitions
to find:
\beq
\xymatrix{
\frF^\sh(F_\Ch(Y, m)) = \frF^\sh(0,Y, Y, \Delta, m) =  (0, Y, Y, \Delta^\sh, m^{-1}) 
}
\eeq
where $\Delta^\sh$ is the  triangle
\beq
\xymatrix{
\Delta^\sh =(Y[-1]\ar[r] & 0\ar[r] & Y\ar[rr]^-{\frakm = \tau^{-1}\circ \canfromid_\Psi} && Y)
}\eeq
as highlighted in Remark~\ref{key rem}.

On the other hand,  we  can also apply the definitions
to find:
\beq
\xymatrix{
F_\Cl(K_\frakm(Y, m)) = \F_\Cl(Y, m^{-1} \circ \frakm) =   (0, Y, Y, \Delta_0^\sh,  m^{-1} \circ \frakm) 
}
\eeq
where $\Delta^\sh_0$ is the triangle
\beq
\xymatrix{
\Delta_0^\sh =(Y[-1]\ar[r] & 0\ar[r] & Y\ar[r]^-{\id} & Y)
}\eeq

Thus we seek an isomorphism
\beq
\xymatrix{
 (0,Y, Y, \Delta^\sh, m^{-1}) \ar[r]^-\sim &
 (0, Y, Y, \Delta^\sh_0, m^{-1}\circ \frakm) 
}
\eeq
or in other words, a map of triangles
\beq
\xymatrix{
\ar@<-1.5em>[d]_-g\Delta^\sh =  (\ar@<1.25em>[d]^-{g_1[-1]}Y[-1]\ar[r] & 0 \ar[d]^-{0} \ar[r] & 
Y\ar[d]^-{g_2} \ar[rr]^-{\frakm = \tau^{-1}\circ \canfromid_\Psi} && Y\ar[d]^-{g_1} )\\
\Delta_0^\sh =  (Y[-1]\ar[r] & 0\ar[r] & Y\ar[rr]^-{\id} && Y )
}\eeq
with a commutativity isomorphism $g_2 \circ m^{-1}  \simeq m^{-1} \circ \frakm \circ g_1$.
We take $g_1 = \id$, $g_2 = \frakm$, and the canonical isomorphism
$g_2 \circ m^{-1} =  \frakm \circ m^{-1}  \simeq m^{-1} \circ \frakm \simeq m^{-1} \circ \frakm  \circ g_1$.
\end{proof}

\begin{corollary}\label{c:trans}
$
K_\frakm \simeq F^{-1}_{\Cl} \circ \frF^\sh \circ F_{\Ch}.
$
\end{corollary}

\begin{example}[Smooth hypersurfaces]\label{ex hyper wc}
Recall from Example~\ref{ex hyper} the perverse schober
\beq
\cP_{Y \subset X} = (i_*:\Coh(Y) \to \Coh(X))
\eeq
Recall framings 
are equivalent to  isomorphisms
\beq
\xymatrix{
\tau:\cO_X \ar[r]^-\sim &  \cL_X
}
\eeq
or in other words, non-vanishing sections of $ \cL_X$.
Recall an object $\cF\in \Coh(X)$ is clean
if and only if $\cF$  is supported away from $Y\subset X$. 
For a clean object $\cF\in \Coh(X)$, we have the central automorphism
\beq
\xymatrix{
\frakm : \cF \ar[r]^-{\sigma} &  \cF\otimes_{\cO_X} \cL_X\ar[r]^-{\tau^{-1}} &  \cF
}
\eeq
We may view $\frakm = \sigma/\tau$ as a function vanishing on $Y=\{\sigma = 0\}$.
\end{example}


\section{Application to toric mutation}

Now let us return to the local model of toric mutations  introduced in Section~\ref{s:local geom} and further discussed in~\ref{s:back to geom}. We will adopt the constructions and notation established therein.

We will derived the wall-crossing formula~\eqref{eq:mut} from the formalism developed in Section~\ref{s: ps},
specifically Corollary~\ref{c:trans}.

As recalled in Section~\ref{s:lgps},
there is a mirror equivalence of perverse schobers $\cP_A \simeq \cP_B$, i.e. a canonically commutative diagram with vertical equivalences
\beq\label{eq:msps}
\xymatrix{
 \ar[d]^-\sim\mu\Sh_{L}(\BC^n)\ar[r]^-{j^*} &  \mu\Sh_{L^\circ}(\BC^n) \ar[d]^-\sim\\
\Perf_\proper(P_{n-2})\ar[r]^-{i_*} &  \Perf_\proper(\BT^\vee_0)
}
\eeq

First, let us extend $\cP_A$ to a perverse schober on the cylinder
by taking the additional monodromy functor to be the identity. Let us write $\cP_{A, \sm}$
for the  perverse schober on the cylinder with the same generic structure but trivial vanishing category.

Then it is an easy exercise, whose proof we sketch, to deduce the following
from Theorem~\ref{thm:lgps}.

\begin{prop}
For the Chekanov and Clifford skeleta $L_{\Ch}, L_\Cl\subset M$, there are natural equivalences
\beq\label{eq:glob sects}
\xymatrix{
  \mu\Sh_{L_\Ch}(M)  \simeq \Gamma(\Cyl, \cP_{A}) &
\mu\Sh_{L_\Cl}(M) \simeq \Gamma^\sh(\Cyl, \cP_{A})  
}
\eeq
Moreover,  for the Chekanov and Clifford  tori $T_\Ch\subset L_{\Ch}, T_\Cl \subset L_\Cl$, the above equivalences restrict to full subcategories to give
\begin{equation*}
\xymatrix{
 \Loc(T_\Ch) \simeq \mu\Sh_{T_\Ch}(M)  \simeq \Gamma(\Cyl, \cP_{A, \sm})&
\Loc(T_\Cl) \simeq \mu\Sh_{T_\Cl}(M)  \simeq \Gamma^\sh(\Cyl, \cP_{A, \sm})
}
\end{equation*}
\end{prop}

\begin{proof}
Let us focus on the first equivalence of \eqref{eq:glob sects}; the second is similar.

Observe that in a small neighborhood $N\subset \BC$ around 
the circle $\gamma_\Ch$, we have an isomorphism of pairs
$(N, \Gamma_\Ch\cap N) \simeq (B^*S^1,  \Lambda\cap B^*S^1)$, where $B^*S^1 \subset T^*S^1$ is a small neighborhood of the zero-section $S^1 \subset T^*S^1$, and $\Lambda\subset T^*S^1$ is the union of the zero-section $S^1 \subset T^*S^1 $ and a single conormal ray $T^+_p S^1 \subset T^*S^1$ based at a point $p\in S^1$.
It is a standard exercise to see 
\beq\label{eq:a1}
\xymatrix{
\mu\Sh_{ \Lambda\cap B^*S^1}(B^*S^1) \simeq \Perf_{\proper}(\BA^1)
}
\eeq

Typically one views objects of $\Perf_{\proper}(\BA^1)$ as objects $Y\in \Perf(pt)$ together with an endomorphism. Equivalently, we can view objects as quadruples $(Y_1, Y_2, \partial, m)$ consisting of objects $Y_1, Y_2\in \Perf(pt)$, a map $\partial:Y_2\to Y_1$, and an isomorphism $m:Y_1\risom Y_2$. Namely, one uses $m$ to identify $Y_2$ and $Y_1$ so that $\partial $ becomes an endomorphism. Note we can normalize the equivalence \eqref{eq:a1} so that restriction to the ray  $T^+_p S^1$ corresponds to forming the shifted cone $\Cone(\partial)[-1]$.

Next, observe that we also have an isomorphism of pairs 
$(W^{-1}(N), L_\Ch\cap W^{-1}(N)) \simeq (B^*S^1,  \Lambda\cap B^*S^1) \times (T^*T^\vanish_\Ch, T^\vanish_\Ch)$,
and hence a natural equivalence 
\beq
\xymatrix{
\mu\Sh_{L_\Ch\cap W^{-1}(N)}(W^{-1}(N)) \simeq \Perf_{\proper}(\BA^1) \otimes \Loc(T^\vanish_\Ch)
}
\eeq
Thus to see the first equivalence of \eqref{eq:glob sects}, one can apply Theorem~\ref{thm:lgps} to see that
$\mu\Sh_{L_\Ch}(M)$ classifies quadruples $(X, Y_1, Y_2, \Delta, m)$ as in the definition of 
$\Gamma(\Cyl, \cP_{A})$. Moreover, the asserted restricted equivalence is evident
\beq
\xymatrix{
\Loc(T_\Ch) \simeq \mu\Sh_{T_\Ch\cap W^{-1}(N)}(W^{-1}(N)) \simeq \Perf_{\proper}(\Gm) \otimes \Loc(T^\vanish_\Ch)
\simeq 
\Loc(T_\Ch)
}
\eeq
\end{proof}

\begin{remark}
Following Remark~\ref{rem:crit},
we could also consider $  \mu\Sh_{L_\crit}(M)$ for the  critical skeleton $L_\crit \subset M$.
It is  again straightforward to construct a natural equivalence  $  \mu\Sh_{L_\crit}(M)  \simeq \Gamma(\Cyl, \cP_{A})$
as a consequence of  Theorem~\ref{thm:lgps}. Here it is less evident how to speak about the full subcategories
corresponding to $\mu\Sh_{T_\Ch}(M) , \mu\Sh_{T_\Cl}(M)$.
\end{remark}

Now let us consider framings for $\cP_A$.

On the one hand, the circle $\gamma_\Ch$ does not wind around $0\in \BC$, and so we have a canonical splitting
$T^\vanish_\Ch \times \gamma_\Ch\risom T_\Ch,
$
and hence an identification for the inertia stack 
\beq
\xymatrix{
 \cL\Loc(T^\vanish_\Ch) \simeq \Loc(T^\vanish_\Ch\times S^1) \simeq 
\Loc(T_\Ch)
}
\eeq

On the other hand, the circle $\gamma_\Cl$ winds once around $0\in \BC$, and so a splitting
$T^\vanish_\Cl \times \gamma_\Cl\risom T_\Cl$ provides a framing $\tau$.
 The choice of a coordinate, say $z_n$, from among $z_1, \ldots, z_n$ gives such a splitting
\beq
\xymatrix{
T^\vanish_\Cl \times \gamma_\Cl \ar[r]^-\sim &  T_\Cl 
& 
((z_1, \ldots, z_n), e^{i\theta})  \ar@{|->}[r] & (z_1, \ldots, z_{n-1}, e^{i\theta} z_n)
}
\eeq
and thus a grading $\tau$.
The splitting in turn provides an identification for the inertia stack 
\beq
\xymatrix{
 \cL\Loc(T^\vanish_\Cl) \simeq \Loc(T^\vanish_\Cl\times S^1) \simeq 
\Loc(T_\Cl)
}
\eeq

\begin{remark}
Under the equivalence $\cP_A \simeq \cP_B$, the framing given by the coordinate $z_n$  corresponds to the
homogenous coordinate section $x_n:\cO \to \cO(1)$.
\end{remark}

More fundamentally, the choice of the coordinate $z_n$ provides an extension 
\beq
\xymatrix{
\ol W:\BC^{n-1}\times\BC\BP^1 \ar[r] & \BC\BP^1 & 
 \ol W = [W, z]
}
\eeq 
where we equip $\BC^{n-1}$ with coordinates $z_1, \ldots, z_{n-1}$, and $\BC\BP^1$ with coordinates $[z_n, z]$.
 
 The circles $\gamma_\Ch, \gamma_\Cl\subset \BC$, with their counterclockwise orientations, are naturally isotopic when regarded within
$\BC\BP^1 \setminus \{0\}\simeq\BC$, but with opposite orientations. 
We also obtain an identification $T^\vanish_\Ch \simeq T^\vanish_\Cl$ by parallel transporting each within $\BC^{n-1}\times\BC\BP^1$  above the respective intervals $[\epsilon - \rho_\Ch, \oo]$, $[-\oo, \epsilon - \rho_\Cl]$
 to the fiber at $\oo \in \BC\BP^1$. 

Thus altogether, we have a canonical identification $T_\Ch \simeq T_\Cl$ compatible with the above splittings, but notably with the inverse map on the factor $S^1$.  
Now with the above identifications, one can trace through the definitions to conclude the following.

\begin{thm} For the perverse schober $\cP_A$, the Chekanov and Clifford functors factor  into the compositions
 \beq
\xymatrix{
F_\Ch:\cL \cD_\invar \ar@{^(->}[r] & \cL\Loc(T^\vanish_\Ch) \simeq
\Loc(T_\Ch) \ar@{^(->}[r] &  \mu\Sh_{L_\Ch}(M)\simeq \Gamma(\Cyl, \cP_A|_{\Cyl})
}
\eeq
 \beq
\xymatrix{
F_\Cl:\cL \cD_\invar \ar@{^(->}[r] & \cL\Loc(T^\vanish_\Cl) \simeq
\Loc(T_\Cl) \ar@{^(->}[r] &  \mu\Sh_{L_\Cl}(M)\simeq \Gamma(\Cyl, \cP_A|_{\Cyl})
}
\eeq
\end{thm}

\begin{corollary}
The wall-crossing formula of Corollary~\ref{c:trans} is the birational map on moduli of objects for the partially defined functor
$
 \Loc(T_\Ch)\to  \Loc(T_\Cl)
$
 given by 
comparing clean local systems on the Chekanov and Clifford tori as objects in  $\Gamma(\Cyl, \cP_A)$
with coordinates related by the framing $\tau$.

\end{corollary}




\end{document}